\documentstyle[amssymb,amsfonts,12pt]{amsart}
\newtheorem{theorem}{Theorem}[section]

\newtheorem{proposition}[theorem]{Proposition}

\newtheorem{conjecture}[theorem]{Conjecture}

\theoremstyle{remark}

\def\QSet{\mbox{\rm\kern.24em
\vrule width.03em height1.48ex depth-.051ex \kern-.26em Q}}

\def\T{{\mathbb T}}

\def\R{{\mathbb R}}

\def\N{{\mathbb N}}
\def\C{{\mathbb C}}

\def\Z{{\mathbb Z}}

\def\F{{\mathcal F}}

\def\k{{\bf k}}
\def\x{{\bf x}}\def\\xi{{\bf \xi}}

\def\be#1{\begin{equation}\label{#1}}

\def\bas{\begin{align*}}
\def\eas{\end{align*}}
\def\bi{\begin{itemize}}
\def\ei{\end{itemize}}
\newenvironment{proof}{\noindent {\bf Proof} }{\endprf\par}
\def \endprf{\hfill  {\vrule height6pt width6pt depth0pt}\medskip}
\def\emph#1{{\it #1}}

\begin{document}

\title[Improved estimates for the  discrete Fourier restriction]{Improved estimates for the  discrete Fourier restriction to the higher dimensional sphere}
\author{Jean Bourgain}
\address{School of Mathematics, Institute for Advanced Study, Princeton, NJ 08540}
\email{bourgain@@math.ias.edu}
\author{Ciprian Demeter}
\address{Department of Mathematics, Indiana University, 831 East 3rd St., Bloomington IN 47405}
\email{demeterc@@indiana.edu}

\thanks{The first  author is supported by a Sloan Research Fellowship and by NSF Grant DMS-1161752}
\begin{abstract}

We improve the exponent in  \cite{Bo1} for the discrete restriction to the $n$ dimensional sphere, from $p=\frac{2(n+1)}{n-3}$ to $p=\frac{2n}{n-3}$, when $n\ge 4$.
\end{abstract}
\maketitle

\section{Introduction}

Let $n\ge 2$ and $\lambda\ge 1$ be two integers. Define $N=[\lambda^{1/2}]+1$ and
$$\F_{n,\lambda}=\{\xi=(\xi_1,\ldots,\xi_n)\in\Z^n:|\xi_1|^2+\ldots|\xi_n|^2=\lambda\}.$$
When $n=2,3,4$ it is known that for each $\epsilon$ we have $|\F_{n,\lambda}|\lesssim N^{n-2+\epsilon}$, but the upper bound is only sharp for certain values of $\lambda$.  For example  $\F_{3,\lambda}=\emptyset$ when $\lambda=4^a(8m+7)$ for $a,m\in\N$.
On the other hand, if $n\ge 5$ we have a sharp estimate $|\F_{n,\lambda}|\sim N^{n-2}$, see \cite{Gr}. Throughout the paper, the implicit bounds hidden in the symbol $\lesssim$ will depend on $\epsilon$, $p$, $q$ and $n$, but never on $N$.

The discrete restriction (sometimes called extension) problem relative to the sphere is concerned with determining the order of growth in $N$ of the numbers
$$M_{p,q,n}(\lambda)=\sup_{a_\\xi\in \C}\frac{\|\sum_{\\xi\in \F_{n,\lambda}}a_\\xie(\\xi\cdot\x)\|_{L^p(\T^n)}}{\|a_\xi\|_{l^q}}$$
for $1\le p,q\le \infty$. We use the notation $e(z)=e^{2\pi i z}$. It is conjectured in \cite{Bo1}, \cite{Bo0} that
\begin{conjecture}
\label{conj1}
For each $n\ge 3$ and $\epsilon>0$ we have
$$M_{p,2,n}(\lambda)\lesssim N^{\frac{n-2}{2}-\frac{n}{p}+\epsilon}$$
for each $p\ge \frac{2n}{n-2}$
\end{conjecture}

The first author proved this in \cite{Bo1}  when $p\ge\frac{2(n+1)}{n-3}$ and $n\ge 4$. Here we improve that range to
\begin{theorem}
\label{thm1}
Assume $n\ge 4$ and $p\ge \frac{2n}{n-3}$. Then for each $\epsilon>0$ we have
$$M_{p,2,n}(\lambda)\lesssim N^{\frac{n-2}{2}-\frac{n}{p}+\epsilon}.$$
\end{theorem}

We would like to thank Yi Hu for stimulating discussions and to Alexandru Zaharescu for pointing out the reference \cite{Sar}.

\section{A brief overview of the known results and methods}

The literature on the discrete restriction to the sphere is very sparse, we are only aware of three relevant papers \cite{Bo0}, \cite{Bo1}, \cite{Bo2}. We start by making a few simple observations.

First, note that $M_{p,q,n}(\lambda)$ is monotone in both $p$ and $q$ and it is always at least 1.  It is conjectured in \cite{Bo1} that for the critical index $p_c:=\frac{2n}{n-2}$ one has
\begin{equation}
\label{e6}
M_{p_c,2,n}(\lambda)\lesssim N^{\epsilon}.
\end{equation}

We recall that the "continuous" analogue of \eqref{e6}, proved by Thomas and Stein, is the estimate
\begin{equation}
\label{T-S}
\|\widehat{fd\sigma}\|_{L^{p}(\R^{n})}\lesssim \|f\|_{L^2(S^{n-1})},\;\;p\ge \frac{2(n+1)}{n-1}.
\end{equation}
The discrepancy between the critical exponents $\frac{2n}{n-2}$ and  $\frac{2(n+1)}{n-1}$ in the discrete and continuous settings can be at least naively explained by the fact that the discrete sphere has "holes". More precisely, $\F_{n,\lambda}$ has roughly $N^{n-2}$ points, while a maximal 1 separated set on the sphere $\{\xi\in\R^n:|\xi|^2=\lambda\}$ has roughly $N^{n-1}$ points. However, this discrepancy is not present in the case of the paraboloid $$\{\xi_n=\xi_1^2+\ldots+\xi_{n-1}^2:-N\le \xi_1,\dots,\xi_{n-1}\le N\},$$
where it is conjectured that $p_c=\frac{2(n+1)}{n-1}$. See \cite{Bo2} for the best known estimate for the paraboloid.

The bound $|\F_{2,\lambda}|\lesssim N^{\epsilon}$ trivially implies $M_{p,q,n}(\lambda)\lesssim N^\epsilon$ when $n=2$, for each $p,q$.
However, \eqref{e6} is open when $n\ge 3$.

On the other hand \eqref{e6} is known for some range below the critical index. For example, the bound for the number of lattice points on ellipses  and a simple counting argument can be easily used to derive the estimate $M_{4,2,3}(\lambda)\lesssim N^{\epsilon}$, see  \cite{Bo2}.
Also, the first authors's recent result in \cite{Bo2} proves \eqref{e6} for $p\le \frac{2n}{n-1}$, $n\ge 2$.

Remarkably, the conjectured bound \eqref{e6} implies  all the correct values $M_{p,q,n}(\lambda)$  within a factor of $N^{\epsilon}$.  This is in contrast with the continuous version of the restriction problem, where the $q=2$ case is fully understood via the work of Thomas and Stein, but a whole range of other estimates remains open (and very difficult!). We prove below that  \eqref{e6} implies
\begin{conjecture}
\label{conj12}
For each $\epsilon>0$ we have
\begin{equation}
\label{hdfguyrfyweuyu3}
M_{p,q,n}\sim N^\epsilon,\text{ if }1\le p\le p_{c,q}:=\frac{q'p_c}{2}\text{ and }q\le2
\end{equation}
\begin{equation}
\label{hdfguyrfyweuyu4}
M_{p,q,n}\sim_{N^{\epsilon}}N^{\frac{(n-2)}{q'}-\frac{n}{p}},\text{ if }p_{c,q}< p\text{ and }q\le2
\end{equation}
\begin{equation}
\label{hdfguyrfyweuyu1}
M_{p,q,n}\sim_{N^{\epsilon}}N^{(n-2)(\frac12-\frac1q)},\text{ if }1\le p\le p_{c}\text{ and }q>2
\end{equation}
\begin{equation}
\label{hdfguyrfyweuyu2}
M_{p,q,n}\sim_{N^{\epsilon}}N^{\frac{(n-2)}{q'}-\frac{n}{p}},\text{ if }p_c<p\text{ and }q>2
\end{equation}
\end{conjecture}
\begin{proof}
We first prove the upper bounds for $M_{p,q,n}$. Note  the trivial estimate
$$M_{\infty,1,n}\le 1.$$ This together with \eqref{e6} implies \eqref{hdfguyrfyweuyu3}, by interpolation. \eqref{hdfguyrfyweuyu4} follows from \eqref{hdfguyrfyweuyu3} and the immediate bound $M_{\infty,q,n}\le N^{\frac{n-2}{q'}}$, via H\"older. To see \eqref{hdfguyrfyweuyu1}, note that using \eqref{e6} and H\"older
$$\|\sum_{\\xi\in \F_{n,\lambda}}a_\\xie(\\xi\cdot\x)\|_{L^p(\T^n)}\le \|\sum_{\\xi\in \F_{n,\lambda}}a_\\xie(\\xi\cdot\x)\|_{L^{p_c}(\T^n)}$$$$\lesssim N^{\epsilon}\|a_\xi\|_{l^2}\le N^{(n-2)(\frac12-\frac1q)+\epsilon}\|a_\xi\|_{l^q}
$$
Finally, to get the upper bound in \eqref{hdfguyrfyweuyu2} note that
$$\|\sum_{\\xi\in \F_{n,\lambda}}a_\\xie(\\xi\cdot\x)\|_{L^p(\T^n)}\lesssim M_{p,2,n}\|a_\xi\|_{l^2}$$$$\le M_{p,2,n}N^{(n-2)(\frac12-\frac1q)}\|a_\xi\|_{l^q}\lesssim N^{\frac{(n-2)}{q'}-\frac{n}{p}+\epsilon}\|a_\xi\|_{l^q},
$$
where the last inequality follows from \eqref{hdfguyrfyweuyu4} with $q=2$.

It remains to prove the lower bounds. The one in \eqref{hdfguyrfyweuyu3} is trivial, by taking the singleton $a_\xi=\delta_{\xi_0}$. Then \eqref{hdfguyrfyweuyu4} and \eqref{hdfguyrfyweuyu2} follow by noticing that
$$ K(\x):=\sum_{\xi\in \F_{n,\lambda}} e(\xi\cdot\x)$$
satisfies $|K(\x)|\gtrsim N^{n-2}$ when $|\x|\lesssim N^{-1}$. Thus $\|K\|_p\gtrsim N^{n-2-\frac{n}{p}}$, while $\|a_\xi\|_{l^q}=N^{\frac{n-2}{q}}$, for each $1\le p,q\le\infty$.

To see \eqref{hdfguyrfyweuyu1},
a standard randomization argument shows that given any $1\le p\le \infty$ and $q>2$, there exists $a_\xi\in\{-1,1\}$ such that
$$\|\sum_{\\xi\in \F_{n,\lambda}}a_\\xie(\\xi\cdot\x)\|_{L^p(\T^n)}\gtrsim \|a_\xi\|_{l^2}= |\F_{n,\lambda}|^{\frac12-\frac1q}\|a_\xi\|_{l^q}.$$
\end{proof}

An immediate corollary is that \eqref{e6} implies Conjecture \ref{conj1}.

We give two slightly different arguments for Theorem \ref{thm1}. The first one seems to only apply to $n\ge 6$ but is technically a bit simpler. The second argument, presented in section \ref{new} covers the full range $n\ge 4$.

In the first argument we apply the point of view from \cite{HL} on the Thomas-Stein restriction argument. This amounts to cutting the kernel in only two pieces, near rationals with denominators greater than $N$. The first piece is small in $L^\infty$ norm. The second piece is supported in frequency away from the sphere, and its Fourier transform is small in the $L^\infty$ norm. This type of construction has a lot of flexibility and in particular allows us to simplify the argument by working with prime moduli. We will rely on three type of level set estimates corresponding to three different regimes. On the one hand, we use the sharp bounds for the Kloosterman and Sali\'e sums, following the approach in \cite{Bo1}. Second we rely on a sharp estimate for certain partial moments of the Weyl sums.
 The third ingredient is the subcritical estimate  in \cite{Bo2}.
It is worth pointing out the  fact that the estimate in \cite{Bo2} does not rely at all on Number Theory, it is entirely of Fourier analytic flavor. See a brief account in Section \ref{CT-sec}.

It seems that a full resolution of the problem would require substantially new insight. One such possible avenue is getting estimates for moments of Kloosterman sums. This is briefly described in the end of the paper. See also \cite{Bo0}.

\section{Some number theoretical generalities}
Let $1_{[-1,1]}\le \gamma\le 1_{[-2,2]}$  be a Schwartz function. Define the smooth Weyl sums
$$G(t,x)=\sum_{k\in \Z}\gamma(k/N)e(kx+k^2t).$$
Inserting the smooth cut off will be completely harmless, in fact it will ease some of our computations.
Let $t=\frac{a}q+\varphi$ where $(a,q)=1$ and $|\varphi|<\frac{1}{q}$. Using the representation $k=rq+k_1$, $0\le k_1\le q-1$ and the Poisson summation formula we get
$$G(t,x)=\sum_{k_1=0}^{q-1}e(k_1^2a/q)\sum_{r\in\Z}\gamma(\frac{k_1+rq}{N})e((rq+k_1)x+(rq+k_1)^2\varphi)$$
$$=\sum_{m\in\Z}\left[\frac1q\sum_{k_1=0}^{q-1}e(k_1^2a/q-k_1m/q)\right]\left[\int_\R\gamma(y/N)e((x+\frac{m}q)y+\varphi y^2)dy\right]$$
\begin{equation}
\label{e2}
=\sum_{m\in\Z}S(a,m,q)J(x,\varphi,m,q)
\end{equation}
where
$$S(a,m,q)=\frac1q\sum_{k=0}^{q-1}e(k^2a/q-km/q)$$
$$J(x,\varphi,m,q)=\int_{\R}\gamma(y/N)e((x+\frac{m}q)y+\varphi y^2)dy.$$

Assume now that $2\le q\le N$, and $|\varphi|\le \frac1{Nq}$. The relevance of this choice is that, according to Dirichlet's theorem every $t\in [0,1]$ is of the form $t=\frac{a}{q}+\varphi$, with $2\le q\le N$ and $|\varphi|\le \frac1{Nq}$. The classical van der Corput estimate reads
$$|\int_{\R}\gamma(z)e(Az+Bz^2)dz|\lesssim |B|^{-1/2},$$
and combining this with the trivial estimate we get
$$|J(x,\varphi,m,q)|\lesssim \min\{N,|\varphi|^{-1/2}\}.$$
On the other hand, repeated integration by parts shows that for each $M$ and $\epsilon$
$$|J(x,\varphi,m,q)|\lesssim_{M,\epsilon} N^{-M}$$
when $|xq+m|\ge N^{\epsilon}$. These values of $m$ will produce a negligible contribution. Combining this with the classical estimate
$$|S(a,m,q)|\lesssim \frac{1}{\sqrt{q}}$$
we get
\begin{equation}
\label{e9}
|G(t,x)|\lesssim_{\epsilon} \frac{N^\epsilon}{\sqrt{q}}\min\{N,|t-\frac{a}q|^{-1/2}\}.
\end{equation}

We will also need more refined estimates for $S(a,m,q)$, in particular we will need to exploit cancelations when summing over $a$. We start by a simple computation.
If $q$ is odd then
$$S(a,m,q)=e(-4^{*}a^{*}m^2/q)\frac1q\sum_{k=0}^{q-1}e(a/q(k^2-2k2^{*}a^{*}m+4^{*}(a^{*})^2m^2)=$$
$$=e(-4^{*}a^{*}m^2/q)\frac1q\sum_{k=0}^{q-1}e(k^2a/q)=e(-4^{*}a^{*}m^2/q)(\frac{a}{q})G(q).$$
Here and in the following, $x^*$ denotes the inverse of $x$ modulo $q$, $(\frac{a}{q})$ is the Jacobi symbol, while
$$G(q)=\frac1q\sum_{k=0}^{q-1}e(k^2/q),$$
is the standard Gauss sum.

Fix $m_j$. Consider the function
$$\Sigma(s)=\sum_{(a,s)=1}\left[\prod_{j=1}^nS(a,m_j,s)e(-\lambda a/s)\right].$$
When $m_j=0$ for each $j$, $\Sigma$ becomes the classical singular series introduced by Hardy and Littlewood in the problem of representations of integers as sums of squares. See for example \cite{Gr} for a detailed discussion.

It is easily seen that $\Sigma$ is multiplicative, though we will not need to exploit this in our argument. Moreover, the previous computations show that for each odd $q$ we have
$$\Sigma(q)={G(q)^n}\sum_{(a,q)=1}(\frac{a}{q})^ne(-\frac{4^{*}\tilde{m}a^{*}}{q}-\lambda \frac{a}q)$$
where $\tilde{m}=m_1^2+\ldots+m_n^2$.

At this point we need to recall the Sali\'e sums, for odd $q$
$$K_2(a,b,q)=\sum_{(k,q)=1}(\frac{k}{q})e(\frac{ka}{q}+\frac{k^{*}b}{q}).$$
If $q$ is a prime number, they have a remarkably simple formula, see for example \cite{Iwa}
$$K_2(a,b,q)=2q\cos(\frac{4\pi x}{q})G(q),$$
where $x^2\equiv ab\pmod q$. In particular, we have
$$|K_2(a,b,q)|\le 2\sqrt{q}$$
for each prime $q$.

Finally, recall the Kloosterman sums
$$K(a,b,q)=\sum_{(k,q)=1}e(\frac{ka}{q}+\frac{k^{*}b}{q}),$$
and their estimates for prime $q$
$$|K(a,b,q)|\lesssim q^{\epsilon}\sqrt{q}\sqrt{\text{gcd}(a,b,q)}.$$

We conclude that for each $q$ prime and for each $n$ (both even and odd) we have
$$|\Sigma(q)|\lesssim q^{\epsilon}(\sqrt{q})^{1-n}\sqrt{(\lambda,q)}.$$

\section{Level set estimates}
If $g:\R^n\to\C$ and $h:\T^n\to\C$, we will denote by $\widehat{g}:\R^n\to\C$ and $\F(h):\Z^n\to\C$ their Fourier transforms.

For $\x\in\T^n$ recall that  $ K(\x)=\sum_{\xi\in \F_{n,\lambda}} e(\xi\cdot\x)$. The integral points on the sphere do not have an explicit formula, we need to introduce a new variable $t$ to fix this deficiency and we notice that
\begin{equation}
\label{e8}
K(\x)=\int_{[0,1]}\prod_{j=1}^n\left[\sum_{k}\gamma(k/N)e(kx_j+k^2t)\right]e(-\lambda t)dt
\end{equation}
The kernel $K$ is the discrete analogue of $\widehat{d\sigma}$, where $d\sigma$ is the surface measure on the sphere $S^{n-1}$ in $\R^n$.

We now proceed with decomposing $K$ in two pieces.
For $N\le Q\le N^2$ define
$$A_Q:=\{Q\le q\le  2Q: q\text{ is prime}\},$$
so that by the Prime Number Theorem we get $|A_Q|\sim Q(\log Q)^{-1}$. The  reason we work with this restricted set of moduli is to simplify the analysis of the Kloosterman, Sali\'e and Ramanujan sums.
The cardinality $N_Q$ of the set of Farrey fractions
$$F_Q:=\{\frac{a}{q}:\;q\in A_Q,\;1\le a\le q-1\}$$
satisfies $N_Q\sim Q^2(\log Q)^{-1}$.

Let  $0\le \eta\le 1_{[-1,1]}$ be a Schwartz function. Define $c_Q=\frac{10Q^2\int\eta}{N_Q}$ and
$$\eta_Q=c_Q\sum_{a/q\in F_Q}\eta((t-{a}/q)10Q^2).$$
Note that $\int \eta_Q=1$ and $c_Q\lesssim \log Q$. Define also
$$K^{Q}(\x)=\int_{[0,1]}\prod_{j=1}^nG(t,x_j)e(-\lambda t)\eta_Q(t)dt.$$

We will prove the following
\begin{proposition}
\label{prop1}
Given $N \le Q\le N^2$ we have for each $n\ge 1$ and $\epsilon$
$$\|K^Q\|_{\infty}\lesssim Q^{\frac{n-1}{2}+\epsilon}$$
\end{proposition}
\begin{proof}
Fix $q\in A_Q$, $|\varphi|\le (10Q^2)^{-1}$ and $\x\in\T^n$. Since $\varphi$ is small, the trivial estimate prevails over the van der Corput one and the best we can say is
$$|J(x_j,\varphi,m,q)|\lesssim N.$$
Repeated integration by parts shows as before that
$$|J(x_j,\varphi,m,q)|\lesssim_{M,\epsilon} N^{-M}$$
if $|x_j+\frac{m}q|>N^{-1+\epsilon}$, for each $\epsilon,M>0$. This means that in the summation \eqref{e2} the range of values of $m$ can be restricted to an interval $I_{x_j,q}$ of length $O(\frac{Q}{N^{1-\epsilon}})$, if error terms of order $O(N^{-M})$ are to be tolerated.

 For each $(m_1,\ldots,m_n)\in \prod I_{x_j,q}$ we have from the previous section that
$$|\sum_{(a,q)=1}\prod_{j=1}^nS(a,m_j,q)e(-\lambda a/q)|\lesssim q^{\epsilon}(\sqrt{q})^{1-n}\sqrt{(\lambda,q)},$$
for each $n\ge 1$.
By invoking \eqref{e2}, summing over the $(Q/N^{1-\epsilon})^n$ values in  $\prod I_{x_j,q}$, and integrating over $|\varphi|\lesssim (10Q^{2})^{-1}$ we get for each $M>0$
$$|K^Q(\x)|\lesssim_{\epsilon,M} Q^{n-2+\epsilon}\sum_{q\in A_Q}q^{\epsilon}(\sqrt{q})^{1-n}\sqrt{(\lambda,q)}+N^{-M}\lesssim Q^{n-2}Q^{\epsilon}(\sqrt{Q})^{3-n}=Q^{\frac{n-1}{2}+\epsilon}.$$
We have used the fact that since $\lambda\le Q^2$, there can be at most one $q\in A_Q$ such that $(\lambda,q)>1$.
\end{proof}

The estimate in the previous proposition is good for $Q$ close to $N$. The next result is a much more elementary estimate which is good for large $Q$.
\begin{proposition}
Given $N \le Q\le N^2$ we have for each $n\ge 4$
$$\|K^Q\|_{\infty}\lesssim N^{2+\epsilon}Q^{\frac{n-4}{2}}$$
\end{proposition}
\begin{proof}
Fix  $x\in \T$ and  $2^s\gtrsim N^\epsilon\sqrt{N}$. From \eqref{e9} we deduce that
$$|\{t\in[0,1]:|G(t,x)|\ge 2^s\}|\lesssim\sum_{q\lesssim (\frac{N^{1+\epsilon}}{2^s})^2}\frac{\phi(q)}{q2^{2s}}\lesssim N^{2+2\epsilon}2^{-4s},$$
where $\phi$ is the Euler totient function.
The proof of Proposition \ref{prop1} shows  that if $t$ is in the support $S_Q$ of $\eta_Q$ we have
$$|G(t,x)|\lesssim N^{\epsilon}\sqrt{Q}.$$
Thus for each fixed $x$
$$\|G(t,x)\|_{L^n(S_Q)}^n\lesssim N^{2+\epsilon}\sum_{N^{\frac12+\epsilon}\le 2^s\lesssim N^{\epsilon}\sqrt{Q}}2^{s(n-4)}+N^{\frac{n}{2}+\epsilon}\lesssim N^{2+\epsilon}Q^{\frac{n-4}{2}}.$$
The result now follows from H\"older in $t$.
\end{proof}

To summarize, we have for each $n\ge 4$
\begin{equation}
\label{3}
\|K^Q\|_{\infty}\lesssim\begin{cases}N^{2}Q^{\frac{n-4}{2}+\epsilon}&: \text{ if } Q\ge N^{4/3}\\  \hfill Q^{\frac{n-1}2+\epsilon}&:\text{if }N\le Q\le N^{4/3}\end{cases}.
\end{equation}

We also have the following estimate on the Fourier side
\begin{proposition}\label{rgrth76i78lo.kj}
Given $N \le Q\le N^2$ we have for each $n\ge 1$
$$\|\F(K-K^Q)\|_{\infty}\lesssim N^{\epsilon}Q^{-1}.$$
\end{proposition}
\begin{proof}
Note that for each $\k\in\Z^n$
$$\F(K-K^Q)(\k)=\widehat{1-\eta_Q}(|\k|^2-\lambda)\prod_{i=1}^n\gamma(\frac{k_i}{N}).$$

If $l$ is any nonzero integer then
$$\widehat{1-\eta_Q}(l)=c_Q(10Q^2)^{-1}\widehat{\eta}(\frac{l}{10Q^2})\sum_{q\in A_Q}\sum_{a=1}^{q-1}e(la/q)$$
$$=c_Q(10Q^2)^{-1}\widehat{\eta}(\frac{l}{10Q^2})\sum_{q\in A_Q}\sum_{a=1}^{q-1}e(la/q)$$
$$=c_Q(10Q^2)^{-1}\widehat{\eta}(\frac{l}{10Q^2})(\sum_{q\in A_Q:q \text{ divides }l}q-|A_Q|).$$
When $l$ gets larger, the increase of the number of its prime divisors from $A_Q$ is offset by the decay of $\widehat{\eta}$
$$|\widehat{\eta}(z)|\lesssim (1+|z|)^{-100}$$
 and we get
$$|\widehat{1-\eta_Q}(l)|\lesssim_{\epsilon} Q^{\epsilon-1}.$$
The result now follows from the fact that $1-\eta_Q$ has mean zero.
\end{proof}

Assume now $\|a_\xi\|_{l^2(\F_{n,\lambda})}=1$ and let
$$F(\x)=\sum_{\xi\in \F_{n,\lambda}}a_\xi e(\xi\cdot\x).$$For  $\alpha>0$ define
$$E_\alpha=\{\x\in\T^n:|F(\x)|>\alpha\},$$
$$f(\x)=\frac{F(\x)}{|F(\x)|}1_{E_{\alpha}}(\x).$$
It follows that
$$\alpha|E_\alpha|\le\int_{\T^n}\bar{F}(\x)f(\x)d\x=\sum_{\xi\in \F_{n,\lambda}}\bar{a_\xi} \F(f)(\xi),$$
and thus
$$\alpha^2|E_\alpha|^2\le\sum_{\xi\in \F_{n,\lambda}}|\F(f)(\xi)|^2=\langle K*f,f\rangle.$$
This in turns implies that
$$\alpha^2|E_\alpha|^2\le \|K^Q\|_{\infty}|E_\alpha|^2+\|\F(K-K^Q)\|_{\infty}|E_\alpha|.$$
We now use  \eqref{3}, by choosing $Q$ appropriately so that the upper bound for $\|K^Q\|_{\infty}$ is roughly $\alpha^2$.  We get for each $n\ge 5$
\begin{equation}
\label{4}
|E_\alpha|\lesssim\begin{cases}N^{\epsilon}\frac{1}{\alpha^{2\frac{n+1}{n-1}}}&: \text{ if } N^{\frac{n-1}{4}}\le \alpha\le N^{\frac{n-1}{3}} \\ \\ \hfill N^{\frac{4}{n-4}+\epsilon}\frac{1}{\alpha^{\frac{2n-4}{n-4}}}&:
\text{ if } \alpha\ge N^{\frac{n-1}{3}}
\end{cases}.
\end{equation}

\section{From continuous to discrete restriction}
\label{CT-sec}
 One may wonder whether the estimate \eqref{T-S} for some $p$ directly implies its discrete analogue, namely  $M_{p,2,n}(\lambda)\lesssim N^\epsilon$. The answer is "no" for both the sphere and the paraboloid, and here is why. It is a basic fact that \eqref{T-S} is  equivalent with ($B_N$ is the ball centered at the origin with radius $N$ in $\R^n$)
$$\|\sum_{\xi\in\Lambda}a_\xi e(\xi\cdot\x)\|_{L^p(B_1)}\lesssim N^{\frac{n-1}{2}-\frac{n}{p}}\|a_\xi\|_{l^2(\Lambda)}$$
for each $a_\xi\in\C$ and each 1-separated set $\Lambda$ on the sphere
$\{\xi\in\R^n:|\xi|^2=\lambda\}$.
The result also holds for the paraboloid
$$\{\xi=(\xi_1,\ldots,\xi_n)\in\R^n:|\xi_1|,\ldots|\xi_{n-1}|\le N,\xi_n=\xi_1^2+\ldots+\xi_{n-1}^2\}.$$
Since \eqref{T-S} fails for $p=\frac{2n}{n-1}$, no valuable information can be derived this way about $M_{\frac{2n}{n-1},2,n}(\lambda)$. Luckily, the index $\frac{2n}{n-1}$ plays a key role in the multilinear restriction theory. More precisely, it was proved in \cite{BCT} that if $P_1,\ldots,P_n$ are transverse regions of the sphere $S^{n-1}$ (or the paraboloid), then one can improve over the Thomas-Stein exponent, at the expense of loosing $N^{\epsilon}$
$$\|(\Pi_{i=1}^n \widehat{fd\sigma_{P_i}})^{1/n}\|_{L^{\frac{2n}{n-1}}(B_N)}\lesssim N^{\epsilon}\|f\|_{L^2(S^{n-1})}.$$
As in the linear case, this implies
$$\|(\Pi_{i=1}^n|\sum_{\xi\in\Lambda_i}a_\xi e(\xi\cdot\x)|)^{1/n}\|_{L^{\frac{2n}{n-1}}(B_1)}\lesssim N^{\epsilon}\|a_\xi\|_{l^2(\Lambda)}$$
where $\Lambda$ is as before, while $\Lambda_i$ are transverse subsets of $\Lambda$. This is the staring point in the argument from \cite{Bo2} which combines it with induction on scales to prove that, if $\Lambda$ is in addition assumed to be in $\Z^n$, we have the unrestricted inequality
$$\|\sum_{\xi\in\Lambda}a_\xi e(\xi\cdot\x)\|_{L^{\frac{2n}{n-1}}(\T^n)}\lesssim N^{\epsilon}\|a_\xi\|_{l^2(\Lambda)}.$$
In particular,
\begin{equation}
\label{BoMo}
M_{\frac{2n}{n-1},2,n}(\lambda)\lesssim N^{\epsilon}.
\end{equation}

\section{Proof of Theorem \ref{thm1}}
We use the notation from the previous section, and the assumption $\|a_\xi\|_{l^2(\F_{n,\lambda})}=1$. Define the index $p_n=\frac{2n}{n-3}$.
The estimate in \eqref{BoMo} implies
$$|E_\alpha|\lesssim N^{\epsilon}\frac{1}{\alpha^{2\frac{n}{n-1}}},$$
valid for each $\alpha>0$ and $n\ge 2$.
Using this,  we get the conjectured bound for each $n\ge 4$ and $p\ge p_n$ for $\alpha$ small
$$\int_0^{N^{\frac{n-1}{4}}} \alpha^{p-1}|E_\alpha|d\alpha\lesssim N^{p(\frac{n-2}{2}-\frac{n}{p}+\epsilon)}.$$

Using the bounds in \eqref{4} we  get
$$\int_{N^{\frac{n-1}{4}}}^{N^{\frac{n-1}{3}}}\alpha^{p-1}|E_\alpha|d\alpha\lesssim N^{p(\frac{n-2}{2}-\frac{n}{p}+\epsilon)},$$
for $n\ge 6$ and $p\ge p_n$, and also
$$\int_{N^{\frac{n-1}{3}}}^{N^{\frac{n-2}{2}}}\alpha^{p-1}|E_\alpha|d\alpha\lesssim N^{p(\frac{n-2}{2}-\frac{n}{p}+\epsilon)},$$
for each $n\ge 5$ and all $p\ge 1$. This completes the proof.

\section{An alternative argument}
\label{new}
We now sketch an alternative argument which will also cover the remaining cases $n=4,5$ of Theorem \ref{thm1}. The argument follows the lines of \cite{Bo3} with input from \cite{Bo1}. Let $\eta$ be an appropriate Schwartz function which equals 1 on $\frac14\le |t|\le \frac12$ and is supported on $\frac18\le |t|\le 1$.
For $Q<N$ and $Q\le 2^s\le N$ we define
$$R_Q=\{\frac{a}{q}:\;(a,q)=1,\;Q\le q<2Q\}$$
$$\eta_{Q,s}(t)=\sum_{a/q\in R_Q}\eta((t-{a}/q)N2^s).$$
Note that $\eta_{Q,s}$ is supported on
$$V_{Q,s}=\{t\in\T:\;|t-\frac{a}{q}|\sim \frac1{N2^s}\text{ for some }\frac{a}q\in R_Q\}.$$
Define also
$$K^{Q,s}(\x)=\int_{[0,1]}\prod_{j=1}^nG(t,x_j)e(-\lambda t)\eta_{Q,s}(t)dt,$$
and the correction factors
$$\rho:=1-\sum_{Q<N/100}\sum_{Q\le 2^s\le N}\eta_{Q,s},$$
$$K^{\text{minor}}=K-\sum_{Q<N/100}\sum_{Q\le 2^s\le N}K^{Q,s}$$

Recall the estimate (2.15) in \cite{Bo1}, (see also Proposition \ref{prop1} here)
\begin{equation}
\label{bnew1}
\|K^{Q,s}\|_\infty\lesssim (N2^s)^{\frac{n}{2}-1+\epsilon}Q^{-\frac{n-3}{2}}.
\end{equation}
An argument very similar to the one in Proposition \ref{rgrth76i78lo.kj} here shows that
\begin{equation}
\label{bnew2}
|\F(K^{Q,s})(\k)|=\begin{cases}\sim \frac{Q^2}{N2^s}&:\quad \text{if}\quad \k=0  \\ \hfill  \lesssim \frac{Q^{1+\epsilon}}{N2^s}&:\quad \text{if}\quad \k\not=0\end{cases}.
\end{equation}
Also, it is immediate that
\begin{equation}
\label{bnew3}
\|K^{\text{minor}}\|_\infty\lesssim N^{\frac{n-1}{2}+\epsilon},
\end{equation}
and
\begin{equation}
\label{bnew4}
|\F(K^{\text{minor}})(\k)|=\begin{cases}=\F(\rho)(0)\sim 1&:\quad \text{if}\quad \k=0  \\ \hfill  \lesssim \frac{1}{N^{1-\epsilon}}&:\quad \text{if}\quad \k\not=0\end{cases}.
\end{equation}

Define
$$\alpha_{Q,s}=\frac{\F(K^{Q,s})(0)}{\F(\rho)(0)}$$
and $K^{Q,s}_1=K^{Q,s}-\alpha_{Q,s}K^{\text{minor}}$.

It follows from \eqref{bnew1}-\eqref{bnew4} that
\begin{equation}
\label{bnew5}
\|\F(K^{Q,s}_1)\|_{\infty}\lesssim \frac{QN^\epsilon}{N2^s}
\end{equation}
and
\begin{equation}
\label{bnew6}
\|K^{Q,s}_1\|_{\infty}\lesssim \frac{(N2^s)^{\frac{n}2-1+\epsilon}}{Q^{\frac{n-3}{2}}}.
\end{equation}
These estimates imply as before that
\begin{equation}
\label{bnew7}
\|F*K^{Q,s}_1\|_{2}\lesssim \frac{QN^\epsilon}{N2^s}\|F\|_2
\end{equation}
and
\begin{equation}
\label{bnew8}
\|F*K^{Q,s}_1\|_{\infty}\lesssim \frac{(N2^s)^{\frac{n}2-1+\epsilon}}{Q^{\frac{n-3}{2}}}\|F\|_1.
\end{equation}
Interpolating between \eqref{bnew7} and \eqref{bnew8} gives for $p_0=\frac{2(n-1)}{(n-3)}$
\begin{equation}
\label{bnew88}
\|F*K^{Q,s}_1\|_{p_0}\lesssim N^{\frac2{n-1}+\epsilon}\|F\|_{p_p'}.
\end{equation}
Thus, if we denote
$$K_1=\sum_{Q<N/100}\sum_{Q\le 2^s\le N}K^{Q,s}_1$$
we also get via the triangle inequality
\begin{equation}
\label{bnew9}
\|F*K_1\|_{p_0}\lesssim N^{\frac2{n-1}+\epsilon}\|F\|_{p_0'}.
\end{equation}
Next we note that
$$
\|\F(K-K_1)\|_{\infty}\lesssim N^{\frac{n-1}{2}+\epsilon}
$$
and thus
\begin{equation}
\label{bnew11}
\|F*(K-K_1)\|_{\infty}\lesssim N^{\frac{n-1}{2}+\epsilon}\|F\|_1.
\end{equation}
Assume now $\|a_\xi\|_{l^2(\F_{n,\lambda})}=1$ and let
$$F(\x)=\sum_{\xi\in \F_{n,\lambda}}a_\xi e(\xi\cdot\x).$$For  $\alpha>0$ define
$$E_\alpha=\{\x\in\T^n:|F(\x)|>\alpha\},$$
$$f(\x)=\frac{F(\x)}{|F(\x)|}1_{E_{\alpha}}(\x).$$
It follows that
$$\alpha|E_\alpha|\le \langle F,f\rangle=\langle F*K,f*K\rangle\le \|f*K\|_2.$$
Thus, by invoking \eqref{bnew9} and \eqref{bnew11} we get
$$
\alpha^2|E_\alpha|^2\le \langle f,f*K\rangle\le |\langle f,f*K_1\rangle|+|\langle f,f*(K-K_1)\rangle|
$$
$$\le \|f\|_{p_0'}^2N^{\frac{2}{n-1}+\epsilon}+\|f\|_1^2N^{\frac{n-1}{2}+\epsilon}$$
\begin{equation*}
\label{bnew12}
\le |E_\alpha|^{\frac2{p_0'}}N^{\frac{2}{n-1}+\epsilon}+|E_\alpha|^2N^{\frac{n-1}{2}+\epsilon}.
\end{equation*}
Thus, for $\alpha>\alpha_0:=N^{\frac{n-1}{4}+\epsilon}$ we get
\begin{equation}
\label{bnew13}
|E_\alpha|\le \alpha^{-2\frac{n-1}{n-3}}N^{\frac2{n-3}}.
\end{equation}
Fix now $p>\frac{2n}{n-3}$. We first use \eqref{BoMo} to write
$$\int |F|^p\le \alpha_0^{p-\frac{2n}{n-1}}+\int_{|F|>\alpha_0} |F|^p.$$
Using \eqref{bnew13} this is further bounded by
$$N^{\frac{n-1}{4}(p-\frac{2n}{n-1})+\epsilon}+N^{\frac{n-2}{2}(p-\frac{2(n-1)}{n-3})+\frac2{n-3}}\lesssim N^{\frac{n-2}{2}p-n}.$$
This finishes the argument.
\section{Closing remarks}

Improving further the range in Theorem \ref{thm1} may rely on exploiting cancelations occurring in sums of Kloosterman sums. Such an example is the Selberg conjecture, which states that
$$|\sum_{q\le X}\frac{K(m,n,q)}{q}|\lesssim (mnX)^{\epsilon}.$$
Since the typical size of $|K(m,n,q)|$ is $\sqrt{q}$, the conjecture predicts a square root cancelation between Kloosterman sums. Recent progress in this direction appears in \cite{Sar} and  \cite{GaSe}.

The piece $K^Q$ of the kernel $K$ introduced earlier in the paper can be defined to incorporate all moduli $Q\le q\le 2Q$ (not only the primes), and the bound in Proposition \ref{rgrth76i78lo.kj} will continue to hold.
It is possible that the correct bound for such a variant of $K^Q$ to be
\begin{equation}
\label{kldfjverf9043i8g56igp[k}
\|K^Q\|_{\infty}\lesssim Q^{\frac{n-2}{2}}.
\end{equation}
This amounts to an additional square root cancelation over the result in Proposition \ref{prop1}. If \eqref{kldfjverf9043i8g56igp[k} held true, the approach described in this paper would imply precisely the sharp level set estimate
$$|E_\alpha|\le \frac{1}{\alpha^{\frac{2n}{n-2}}},$$
albeit only for $\alpha\gtrsim N^{\frac{n-2}{4}}.$
The difficulty of getting the estimate
$$|K^Q(\x)|\lesssim Q^{\frac{n-2}{2}}$$
for a fixed $\x$ comes from the fact that while one of the entries $m,n$ in the Kloosterman sum is fixed (it equals $-\lambda$), the other entry is variable, it depends on $q$.

We also mention that appropriate control over sums of Kloosterman sums would allow a circle method treatment of the
representation problem of integers by sums of three squares.

\end{document}